\documentclass[11pt]{article}

\usepackage[a4paper,margin=1.15in]{geometry}
\usepackage[T1]{fontenc}
\usepackage{lmodern}
\usepackage{amsmath,amssymb,amsthm,mathtools,mathrsfs}
\usepackage{enumitem}
\usepackage{microtype}
\usepackage[hidelinks]{hyperref}
\usepackage{cleveref}
\usepackage{tikz}
\usetikzlibrary{arrows.meta,calc,decorations.pathreplacing}
\usepackage{booktabs}

\newtheorem{theorem}{Theorem}[section]
\newtheorem{lemma}[theorem]{Lemma}
\newtheorem{corollary}[theorem]{Corollary}
\newtheorem{proposition}[theorem]{Proposition}

\theoremstyle{remark}
\newtheorem{remark}[theorem]{Remark}
\newtheorem*{bassoquestion}{Basso's Question 25}

\newcommand{\M}{\mathbf M}
\newcommand{\I}{\mathbf I}
\newcommand{\N}{\mathbf N}
\newcommand{\Hh}{\mathcal H}
\newcommand{\set}{\operatorname{set}}
\newcommand{\spt}{\operatorname{spt}}

\newcommand{\diam}{\operatorname{diam}}
\newcommand{\Odd}{\operatorname{Odd}}
\newcommand{\llbracket}{[\![}
\newcommand{\rrbracket}{]\!]}

\title{\bf Odd branching obstructs multiplicity-one integral current structures}

\author{%
Deguang Zhong\textsuperscript{1}\\
\small\textsuperscript{1}Institute of Applied Mathematics, Shenzhen Polytechnic University,\\[-0.15em]
Shenzhen 518055, China\\
Email: \texttt{huachengzhon@163.com}
}

\date{13 September 2026}

\begin{document}

\maketitle

\begin{abstract}
Giuliano Basso asked whether every $n$-dimensional integral current space admits an integral current structure with the same characteristic set and multiplicity one almost everywhere. We give a negative answer already in dimension one. The obstruction is a parity phenomenon at odd-valence branch points. More precisely, if a one-dimensional integral current has unit multiplicity on the arms incident to an isolated vertex $v$, then the coefficient of its boundary at $v$ has the same parity as $\deg(v)$. Hence infinitely many odd-valence vertices force infinite boundary mass for every full-support unit-multiplicity current.

We construct a compact geodesic, doubling, $1$-Ahlfors regular graph-like continuum $G$ of finite length and a boundaryless current $T\in\I_1(G)$ with $\set(T)=G$, while no $S\in\I_1(G)$ with $\set(S)=G$ can have multiplicity one $\Hh^1$-almost everywhere. We classify all full-support integral cycles on $G$. In particular, the least possible essential maximal multiplicity is exactly two, and the minimum mass of such a cycle is $3/2$. For finite metric graphs we identify a complementary sharp defect: the minimum boundary mass among unit-multiplicity currents equals the number of odd-degree vertices. Finally, by taking products with flat tori and using the precise Ambrosio--Kirchheim slice representation, we obtain boundaryless compact geodesic Ahlfors-regular counterexamples in every dimension.
\end{abstract}



\section{Introduction}

Metric currents were introduced by Ambrosio and Kirchheim \cite{AK00}, extending the Federer--Fleming theory \cite{FedererFleming60,Federer69,Whitney57} to arbitrary complete metric spaces. Kirchheim's metric rectifiability theory \cite{Kirchheim94} and the general metric-measure background in \cite{AmbrosioTilli04} underlie the rectifiable representation. Sormani and Wenger subsequently introduced integral current spaces as metric spaces equipped with a completely settled integral current structure \cite{SW11}; see also \cite{PortegiesSormani18,Lang11,LangWenger11} for related developments, local currents, and compactness. Compactness and approximation tools for metric integral currents have been developed further by Wenger \cite{Wenger11} and Goldhirsch \cite{Goldhirsch22}. Recent work also uses current structures in filling and Finsler-volume problems; see, for example, \cite{BassoFinsler25, BassoCreutzSoultanis23}.

In a 2025 list of open questions in metric geometry and functional analysis, Basso asked the following \cite{BassoQuestions25}.

\begin{bassoquestion}
Let $X$ be an $n$-dimensional integral current space. Does there exist an integral $n$-current $S$ in the metric completion of $X$ such that
\[
    \set(S)=X
\]
and the multiplicity of $S$ is equal to one $\Hh^n$-almost everywhere?
\end{bassoquestion}

Basso remarks that the question may perhaps already be implicit in the foundational literature on integral current spaces. The January 3, 2025 version of the problem list presents it as a question. As of September 2026, we are not aware of a published or publicly available preprint answering it; accordingly, our priority statement is limited to the literature we could locate. The answer is negative in full generality, for an elementary reason that is invisible on manifolds: odd branching creates an unavoidable boundary defect for unit multiplicity.

The possibility of multiplicity one is important in more regular geometric settings. For example, Jaramillo, Perales, Rajan, Searle and Siffert construct a weight-one, boundaryless integral current structure on every closed orientable Alexandrov space \cite{JPRSS21}, building on the interaction between Alexandrov geometry, orientability and current homology; compare also Mitsuishi's identification of integral-current and singular homology on locally Lipschitz contractible spaces \cite{Mitsuishi19}. Matveev and Portegies proved multiplicity-one and constancy conclusions for noncollapsing Ricci-limit spaces under uniform geometric hypotheses \cite{MatveevPortegies17}; Honda's work on orientability of Ricci-limit spaces gives further current-theoretic context \cite{Honda17}. For metric manifolds, Basso, Marti and Wenger constructed an essentially unique top-dimensional current under weak quantitative topological hypotheses and developed the metric fundamental-class viewpoint \cite{BassoMartiWenger25}; Marti and Soultanis subsequently characterized the existence of such fundamental classes under finite Nagata dimension hypotheses \cite{MartiSoultanis24}. Recent extensions to non-orientable manifolds and manifolds with boundary, and further work on current homology, are due to Marti \cite{Marti26Fundamental,Marti26Homology}. Thus multiplicity one is a genuine structural conclusion under additional assumptions, rather than a formal consequence of the definition of an integral current space. This contrast suggests a more structured version of Basso's question: which local geometric or topological hypotheses guarantee a full-support unit-multiplicity current? The examples below show that any such criterion must, at a minimum, exclude persistent codimension-one odd branching.

Our first result isolates the local obstruction. A precise version is proved in \cref{sec:parity}.

\begin{theorem}\label{thm:parity-intro}
Let $G$ be a metric graph-like space, let $v\in G$ be an isolated vertex of finite degree $d$, and let $S\in\I_1(G)$. Suppose that, on each open edge segment incident to $v$, the multiplicity of $S$ is one almost everywhere. If
\[
    \partial S=\sum_{j=1}^N m_j\llbracket x_j\rrbracket,
\]
then the coefficient $m_v$ of $\partial S$ at $v$ satisfies
\[
    m_v\equiv d\pmod 2.
\]
In particular, if $d$ is odd, then $m_v\neq 0$.
\end{theorem}

The theorem is essentially a metric-current form of the elementary parity principle behind Euler trails. Its usefulness here comes from the fact that a zero-dimensional integral current of finite mass has only finitely many nonzero integer atoms. Thus infinitely many isolated odd-valence vertices are incompatible with a full-support unit-multiplicity integral $1$-current.

We then build a compact finite-length graph-like continuum containing infinitely many trivalent vertices but still carrying a boundaryless integral current if multiplicity two is allowed on selected edges.

\begin{theorem}[Main counterexample]\label{thm:main-intro}
There exist a compact metric space $G$ and a current $T\in\I_1(G)$ such that:
\begin{enumerate}[label=\textup{(\roman*)}]
\item $G$ is geodesic, doubling, and $1$-Ahlfors regular;
\item $\partial T=0$ and $\set(T)=G$;
\item there is no $S\in\I_1(G)$ satisfying $\set(S)=G$ and having multiplicity one $\Hh^1$-almost everywhere.
\end{enumerate}
Consequently, Basso's Question~25 has a negative answer already in dimension one.
\end{theorem}

The same example admits a sharp analysis of all of its full-support cycles. If
\[
 \mathscr Z(G)=\{R\in\I_1(G):\partial R=0,\ \set(R)=G\},
\]
then the classification in \cref{thm:cycle-classification} gives
\begin{equation}\label{eq:intro-sharp}
 \min_{R\in\mathscr Z(G)}
 \operatorname*{ess\,sup}_{G}\theta_R=2,
 \qquad
 \min_{R\in\mathscr Z(G)}\M(R)=\frac32.
\end{equation}
Thus the obstruction is quantitatively sharp: multiplicity one is impossible, but multiplicity two suffices.

For finite graphs the same parity argument yields a sharp formula rather than a nonexistence statement. If $K$ is a finite metric graph, let $\mathscr U(K)$ be the class of integral $1$-currents that have absolute multiplicity one on the interior of every edge. Then
\[
    \inf_{S\in\mathscr U(K)}\M(\partial S)=\#\{v\in V(K):\deg(v)\text{ is odd}\}.
\]
The upper bound is the classical balanced-orientation construction obtained from Euler's theorem; see, for example, \cite{Diestel17}. Thus the obstruction is not branching by itself, but an \emph{infinite nonsummable odd-branching defect}.

Finally, we pass to all dimensions by taking products with flat tori. The product construction is obtained from Wenger's interval product \cite[Section 3.1]{Wenger07} by identifying opposite faces. The resulting currents are cycles, and the nonexistence assertion follows by combining a regular finite-to-one torus coordinate map with the slicing theorems of Ambrosio--Kirchheim \cite[Theorems 5.6--5.7 and 9.7]{AK00}.

\begin{theorem}[Counterexamples in every dimension]\label{thm:all-dim-intro}
For every integer $n\ge 1$ there exist a compact, geodesic, doubling, $n$-Ahlfors regular metric space $X_n$ and a current $T_n\in\I_n(X_n)$ such that
\[
    \partial T_n=0,
    \qquad
    \set(T_n)=X_n,
\]
but no $S\in\I_n(X_n)$ satisfies both $\set(S)=X_n$ and $\theta_S=1$ $\Hh^n$-almost everywhere.
\end{theorem}

The proofs use only standard facts from the theory of metric currents: the representation of zero-dimensional integral currents, one-dimensional constancy on intervals, closure and slicing, and elementary graph orientation. The relation between the parity argument and one-dimensional current decompositions is recorded in Remark \ref{rem:decomposition}.

\section{Preliminaries on metric currents and metric graphs}\label{sec:prelim}

We recall only the facts used below. Our conventions follow Ambrosio--Kirchheim \cite{AK00}. For a modern exposition of local currents and slicing, see Lang \cite{Lang11}.

\subsection{Integral currents and characteristic sets}

Let $Z$ be a complete metric space. We write $\I_k(Z)$ for the group of $k$-dimensional integral currents in $Z$, $\M(T)$ for the mass of a current, $\|T\|$ for its mass measure, and $\partial T$ for its boundary.

If $T\in\I_k(Z)$, its canonical or characteristic set is
\[
 \set(T)
 =\left\{x\in Z:
 \Theta_{*k}(\|T\|,x)>0
 \right\},
\]
where
\[
 \Theta_{*k}(\mu,x)
 =\liminf_{r\downarrow0}
   \frac{\mu(B(x,r))}{\omega_k r^k}.
\]
An integral current space in the sense of Sormani--Wenger is a triple $(X,d,T)$ with $T\in\I_n(\overline X)$ and $\set(T)=X$ \cite{SW11}.

For an integer rectifiable current $T$, we write $\theta_T$ for the intrinsic nonnegative integer multiplicity in the Ambrosio--Kirchheim rectifiable representation. The representation and its independence of the ambient isometric embedding are given in \cite[Theorems 9.5--9.6]{AK00}. Thus the statement ``$\theta_T=1$ $\Hh^n$-almost everywhere on $X$'' is intrinsic whenever $\set(T)=X$.

We will also use the standard identification of currents concentrated on a closed isometric subspace. If $Y\subset Z$ is closed and $R\in\I_k(Z)$ has $\|R\|(Z\setminus Y)=0$, then $R=i_\#R_Y$ for a unique $R_Y\in\I_k(Y)$, where $i:Y\hookrightarrow Z$ is inclusion. This follows by extending Lipschitz test functions from $Y$ to $Z$ and using locality; compare the push-forward and restriction formalism in \cite{AK00,Lang11}.

We repeatedly use the elementary form of zero-dimensional integral currents.

\begin{lemma}\label{lem:I0}
If $R\in\I_0(Z)$, then
\[
   R=\sum_{j=1}^{N} m_j\llbracket x_j\rrbracket
\]
for finitely many points $x_j\in Z$ and nonzero integers $m_j$, and
\[
   \M(R)=\sum_{j=1}^{N}|m_j|.
\]
\end{lemma}

\begin{proof}
Ambrosio--Kirchheim's description of integer rectifiable $0$-currents gives a countable atomic representation with integer coefficients and summable absolute values; see \cite[Section 4]{AK00} or \cite[Section 2]{Wenger07}. Since every nonzero integer coefficient has absolute value at least one, finite mass forces the support to be finite.
\end{proof}

\begin{remark}\label{rem:integer-atoms}
The finiteness conclusion in Lemma \ref{lem:I0} is specific to integral currents. A finite signed measure may have infinitely many atoms with summable masses, whereas every nonzero atom of an integral $0$-current has mass at least one. This observation is the global finiteness mechanism used in Corollary \ref{cor:infinite-odd} and Theorem \ref{thm:negative1}.
\end{remark}

We will also use the following one-dimensional constancy fact. A detailed proof is included because it is the only local current-theoretic input in the parity argument.

\begin{lemma}\label{lem:interval}
Let $I=(a,b)\subset\mathbb R$ and let $R\in\I_1(I)$. If $\partial R=0$ in $I$, then
\[
      R=m\llbracket I\rrbracket
\]
for some $m\in\mathbb Z$. If, in addition, the multiplicity of $R$ has absolute value one almost everywhere, then $m\in\{-1,1\}$.
\end{lemma}

\begin{proof}
A top-dimensional integer rectifiable current in an interval is represented by an integer-valued $L^1$ function $u$, namely $R=\llbracket u\rrbracket$. Normality says that the distributional derivative $Du$ is a finite measure, and $\partial R=0$ in $I$ says $Du=0$ there. Hence $u$ is almost everywhere equal to a constant integer. The last assertion follows from $|u|=1$ almost everywhere. This is the one-dimensional instance of the classical constancy principle; compare \cite{Federer69,AK00,Smirnov94,PaoliniStepanov13,BDP22}.
\end{proof}

\subsection{Metric graphs}

By a \emph{metric graph-like space} in this paper we mean the metric completion of a finite or countable one-dimensional length complex obtained by gluing compact intervals along endpoints, with pairwise disjoint edge interiors, whenever the completion is the space under discussion. This convention deliberately includes the completed shrinking graph constructed in \cref{sec:graph}; it does not assert local finiteness at accumulation points. At every isolated finite-degree vertex used below there is, by assumption, an ordinary finite star neighborhood. An oriented edge $e$ of length $\ell(e)$ determines an integral current $\llbracket e\rrbracket$ by pushing forward the standard current on $[0,\ell(e)]$ under the arclength parametrization. Then
\[
  \M(\llbracket e\rrbracket)=\ell(e),
  \qquad
  \partial\llbracket e\rrbracket
  =\llbracket e_+\rrbracket-\llbracket e_-\rrbracket.
\]

For graph-like spaces, finite-length graph-like continua, cycle decompositions, and Eulerian theory, see \cite{ThomassenVella08,Richter11,Georgakopoulos12,Georgakopoulos14,EspinozaGartsidePitz20,GartsidePitz23}. We use only the elementary metric-graph structure of the explicit example constructed below.

\section{The odd-valence parity obstruction}\label{sec:parity}

We now prove the local parity theorem. The key point is that integrality turns parity into a lower bound for boundary mass.

\begin{theorem}\label{thm:local-parity}
Let $G$ be a metric graph-like space, let $v$ be an isolated vertex of finite degree $d$, and let $S\in\I_1(G)$. Assume that there exists a star neighborhood $U$ of $v$ such that on every connected component of $U\setminus\{v\}$ the multiplicity of $S$ has absolute value one almost everywhere. Write
\[
   \partial S=\sum_{j=1}^{N}m_j\llbracket x_j\rrbracket
\]
as in \cref{lem:I0}, and let $m_v=0$ if $v\notin\{x_1,\dots,x_N\}$. Then
\[
      m_v\equiv d\pmod 2.
\]
Consequently, odd degree implies $|m_v|\ge1$.
\end{theorem}

\begin{proof}
By \cref{lem:I0}, $\spt(\partial S)$ is finite. Shrinking $U$ if necessary, we may suppose that $v$ is the only graph vertex in $U$ and the only possible point of $\spt(\partial S)$ in $U$. Choose $\rho>0$ so that the closed $\rho$-star around $v$ is contained in $U$ and consists of $d$ geodesic arms
\[
   \gamma_i:[0,\rho]\longrightarrow G,\qquad \gamma_i(0)=v,
\]
with pairwise disjoint interiors. Orient every arm away from $v$.

On the open part of each arm the boundary of $S$ vanishes. Restricting locally to the arm and pulling back by the isometry $\gamma_i$ gives a locally integral current on $(0,\rho)$ with zero boundary there. Applying Lemma \ref{lem:interval} on relatively compact subintervals and using compatibility on overlaps, its coefficient is constant on the whole open arm. The unit-multiplicity hypothesis therefore gives a sign $\varepsilon_i\in\{-1,1\}$ such that, as a local current on $\gamma_i((0,\rho))$,
\[
  S
  =\varepsilon_i\,\gamma_{i\#}\llbracket(0,\rho)\rrbracket.
\]
This use of restriction to an open set is standard in the local-current formalism; see \cite{Lang11}. There is no nonzero one-dimensional rectifiable current concentrated at the single point $v$, since the mass measure of a rectifiable $1$-current is absolutely continuous with respect to $\Hh^1$ on its rectifiable carrier \cite[Theorem~4.6]{AK00}.

Choose a Lipschitz function $\varphi:G\to[0,1]$ supported in the open $\rho$-star, with $\varphi(v)=1$ and with $\varphi\circ\gamma_i$ decreasing to $0$ before $t=\rho$ on every arm. Since $v$ is the only possible boundary atom in the support of $\varphi$,
\[
     \partial S(\varphi)=m_v.
\]
On the other hand, by the current boundary formula and the preceding arm representation,
\[
\begin{aligned}
  m_v
  &=S(1,\varphi)
    =\sum_{i=1}^{d}\varepsilon_i
      \int_0^\rho (\varphi\circ\gamma_i)'(t)\,dt\\
  &=-\sum_{i=1}^{d}\varepsilon_i.
\end{aligned}
\]
Reducing modulo two and using $\varepsilon_i\equiv1\pmod2$ yields
\[
     m_v\equiv d\pmod2.
\]
If $d$ is odd, $m_v$ is a nonzero odd integer.
\end{proof}

\begin{corollary}\label{cor:infinite-odd}
Let $G$ be a metric graph-like space containing infinitely many distinct isolated vertices $v_k$ of finite odd degree. There is no current $S\in\I_1(G)$ whose multiplicity has absolute value one almost everywhere on every edge incident to every $v_k$.
\end{corollary}

\begin{proof}
By Theorem \ref{thm:local-parity}, every $v_k$ is a nonzero atom of $\partial S$. This contradicts Lemma \ref{lem:I0}.
\end{proof}

We now record the corresponding statement for flat chains modulo $2$. We write
$[R]_2$ for reduction modulo $2$, $\mathcal F_2$ for the mod-$2$ flat norm,
and $\M_2$ for the relaxed mod-$2$ mass, in the sense of Ambrosio--Katz
\cite[Sections~6 and 10]{AmbrosioKatz11}. Thus flat chains modulo $2$ form the
$\mathcal F_2$-completion, the boundary operator is $\mathcal F_2$-continuous,
and an integral flat chain modulo $2$ has finite mod-$2$ mass together with
finite mod-$2$ boundary mass.

\begin{proposition}[The canonical mod-$2$ obstruction]\label{prop:mod2}
Let $H$ be a graph-like space with countably many edges
$E(H)=\{e_1,e_2,\ldots\}$ such that
\[
       \sum_{e\in E(H)}\ell(e)<\infty,
\]
and suppose that the complement of the open edge interiors is
$\Hh^1$-null. Give each edge an arbitrary orientation and let
\[
       C_H=\sum_{e\in E(H)}[\llbracket e\rrbracket]_2,
\]
where the series is understood in mod-$2$ mass, and hence also in the
mod-$2$ flat norm. Then $C_H$ is a rectifiable flat chain modulo $2$ of
finite mod-$2$ mass and is independent of the chosen orientations. At every
isolated vertex $v$ of finite degree, the local boundary coefficient of
$C_H$ is
\[
       \deg_H(v)\pmod2.
\]
Consequently, if $H$ has infinitely many isolated odd-degree vertices, then
\[
       \M_2(\partial C_H)=\infty;
\]
in particular, $C_H$ is not an integral flat chain modulo $2$. Moreover, any $S\in\I_1(H)$ having absolute multiplicity one almost everywhere on every open edge would satisfy $[S]_2=C_H$.
\end{proposition}

\begin{proof}
For
\[
       C_N=\sum_{j=1}^{N}[\llbracket e_j\rrbracket]_2
\]
and $M>N$, subadditivity of the relaxed mass gives
\[
 \mathcal F_2(C_M-C_N)
 \le \M_2(C_M-C_N)
 \le \sum_{j=N+1}^{M}\ell(e_j).
\]
Thus $(C_N)$ is Cauchy in mod-$2$ mass and in the mod-$2$ flat norm; denote
its flat limit by $C_H$. The countable edge representation, together with
the disjointness of the open edge interiors and the nullity of their
complement, shows directly that $C_H$ is rectifiable and that
\[
       \M_2(C_H)\le \sum_{e\in E(H)}\ell(e)<\infty.
\]
More precisely, lower semicontinuity of $\M_2$ under flat convergence gives
\[
       \M_2(C_H-C_N)
       \le \sum_{j>N}\ell(e_j)\longrightarrow0,
\]
so the displayed series does converge in mod-$2$ mass as asserted.
Reversing the orientation of an edge changes its integer coefficient from
$1$ to $-1$, which has no effect modulo $2$. Hence the limit is independent
of all orientation choices.

The estimate
\[
       \mathcal F_2(\partial A)\le \mathcal F_2(A)
\]
shows that $\partial C_N\to\partial C_H$ in the mod-$2$ flat norm. Let $v$
be an isolated vertex of degree $d$. All the finitely many edges incident to
$v$ occur in some $C_N$. Choose a star neighborhood $U_v$ containing no
other vertex and meeting only those incident edges. On $U_v$ the tail of the
series vanishes, so the local boundary of $C_H$ agrees with that of the
finite sum of its $d$ incident edge chains. Each incident oriented edge has
coefficient $1$ or $-1$ at $v$, and both reduce to $1$ modulo $2$. Hence the
local coefficient of $\partial C_H$ at $v$ is $d$ modulo $2$.

Now let $F$ be any finite set of isolated odd-degree vertices. Shrinking the
star neighborhoods, they may be chosen pairwise disjoint. In each such
neighborhood $\partial C_H$ has a nonzero $0$-dimensional coefficient and
therefore mod-$2$ mass at least one. Consequently
\[
       \M_2(\partial C_H)\ge \#F.
\]
If there are infinitely many isolated odd-degree vertices, finite sets $F$
of arbitrarily large cardinality give
$\M_2(\partial C_H)=\infty$. Thus $C_H$ has no finite-mass integral boundary
modulo $2$ and is not an integral flat chain modulo $2$.

Finally, suppose that $S\in\I_1(H)$ has absolute multiplicity one almost
everywhere on every open edge. In the rectifiable representation of $S$ on
an edge, its signed integer coefficient is therefore $1$ or $-1$ almost
everywhere; it need not be constant if $\partial S$ has atoms in the edge
interior, but its reduction modulo $2$ is nevertheless identically one.
Thus $[S]_2$ and $C_H$ agree on the union of the open edge interiors. Both
one-dimensional mass measures vanish on the complementary $\Hh^1$-null
set, so $[S]_2=C_H$. Since reduction commutes with the boundary,
\[
       \partial C_H=\partial[S]_2=[\partial S]_2,
\]
whose mod-$2$ mass is finite because $\partial S\in\I_0(H)$. This recovers
the contradiction when infinitely many isolated odd vertices are present.
\end{proof}

\begin{remark}\label{rem:decomposition}
The parity obstruction is also consistent with the decomposition theory of one-dimensional currents. Smirnov decomposes solenoidal one-dimensional flows in Euclidean space into elementary solenoids \cite{Smirnov94}; Paolini and Stepanov obtain curve and cycle decompositions for normal currents in metric spaces \cite{PaoliniStepanov12,PaoliniStepanov13}; Bonicatto, Del Nin and Pasqualetto characterize the one-dimensional indecomposable integral currents in terms of injective Lipschitz curves and loops \cite{BDP22}; and recent $SBV$ representations of one-dimensional metric currents are developed by Arroyo-Rabasa and Bouchitt\'e \cite{ArroyoRabasaBouchitte25}. From this viewpoint, a unit of oriented flow cannot pass through infinitely many trivalent vertices without producing infinitely many endpoints. The direct proof of Theorem \ref{thm:local-parity} isolates this conclusion without invoking a global decomposition.
\end{remark}

\section{The sharp defect on finite metric graphs}\label{sec:finite}

For a finite connected metric graph $K$ with at least one edge, let $V(K)$ and $E(K)$ be its vertex and edge sets; loops, if present, are counted twice in the degree. Let
\[
   \Odd(K)=\{v\in V(K):\deg_K(v)\text{ is odd}\}.
\]
Define $\mathscr U(K)$ to be the class of currents $S\in\I_1(K)$ whose multiplicity has absolute value one $\Hh^1$-almost everywhere on every edge. Such a current automatically has characteristic set $K$.

\begin{theorem}\label{thm:finite-defect}
For every finite connected metric graph $K$,
\[
   \inf_{S\in\mathscr U(K)}\M(\partial S)
   =\#\Odd(K).
\]
The infimum is attained by a current whose orientation is balanced at every vertex: the boundary coefficient is $0$ at even-degree vertices and has absolute value $1$ at odd-degree vertices.
\end{theorem}

\begin{proof}
The lower bound follows from Theorem \ref{thm:local-parity}: every odd-degree vertex contributes a nonzero integer atom to $\partial S$, hence at least one unit of boundary mass.

For the upper bound, pair the vertices of $\Odd(K)$; the handshaking lemma guarantees that their number is even. Add one auxiliary edge for each pair. Every vertex of the resulting finite multigraph has even degree, so every connected component admits an Euler circuit. Orient each original and auxiliary edge along such circuits and then delete the auxiliary edges. At every original even-degree vertex the indegree and outdegree agree, whereas at every original odd-degree vertex they differ by exactly one. Assigning multiplicity one to every oriented original edge gives $S\in\mathscr U(K)$ with
\[
   \M(\partial S)=\#\Odd(K).
\]
This is the standard balanced-orientation consequence of Euler's theorem; compare \cite{Diestel17}.
\end{proof}

\begin{remark}\label{rem:finite-positive}
It is convenient to write
\[
   \mathfrak b(K)
   :=\inf\{\M(\partial S):S\in\mathscr U(K)\}.
\]
Then \cref{thm:finite-defect} says precisely that
\[
   \mathfrak b(K)=\#\Odd(K)
\]
for every finite metric graph. In particular, every finite metric graph admits a full-support unit-multiplicity integral current; odd vertices force boundary, but only of finite mass. For graph-like continua of finite length, it is natural to seek a criterion for the finiteness of the analogous defect. Infinite Eulerian theory and cycle-space methods \cite{ThomassenVella08,Richter11,Georgakopoulos12,Georgakopoulos14,EspinozaGartsidePitz20,GartsidePitz23} indicate that ends and accumulation points must enter alongside the parity of isolated finite-degree vertices.
\end{remark}

\section{A compact Ahlfors-regular graph with infinitely many trivalent vertices}\label{sec:graph}

We now construct the underlying metric space of the counterexample.

Set
\[
   \lambda_k=2^{-k-2},\qquad k\ge1.
\]
Start with vertices
\[
   o,a_1,b_1,a_2,b_2,\ldots.
\]
Join $o$ to $a_1$ by an edge $t_0$ of length $1/4$. For every $k\ge1$, join $a_k$ to $b_k$ by two distinct parallel edges $u_k$ and $v_k$, each of length $\lambda_k$, and join $b_k$ to $a_{k+1}$ by an edge $t_k$ of length $\lambda_k$. The forward chain has a unique completion point $z$. Finally attach an edge $c$ of length $1/4$ from $z$ back to $o$. Let $G$ be the resulting path-metric space.

\begin{figure}[t]
\centering
\begin{tikzpicture}[scale=0.95, every node/.style={font=\small}]
\coordinate (o) at (0,0);
\coordinate (a1) at (1.7,0);
\coordinate (b1) at (3.1,0);
\coordinate (a2) at (4.2,0);
\coordinate (b2) at (5.25,0);
\coordinate (a3) at (6.05,0);
\coordinate (b3) at (6.85,0);
\coordinate (z) at (8.5,0);
\draw[-{Latex[length=2mm]}] (o)--(a1) node[midway,below] {$t_0$};
\foreach \A/\B/\lab in {a1/b1/1,a2/b2/2,a3/b3/3}{
  \draw (\A) to[bend left=28] node[midway,above] {$u_{\lab}$} (\B);
  \draw (\A) to[bend right=28] node[midway,below] {$v_{\lab}$} (\B);
}
\draw[-{Latex[length=2mm]}] (b1)--(a2) node[midway,below] {$t_1$};
\draw[-{Latex[length=2mm]}] (b2)--(a3) node[midway,below] {$t_2$};
\draw[densely dotted] (b3)--(z);
\draw (z) to[bend left=45] node[midway,above] {$c$} (o);
\foreach \P/\N in {o/$o$,a1/$a_1$,b1/$b_1$,a2/$a_2$,b2/$b_2$,a3/$a_3$,b3/$b_3$,z/$z$}{
  \fill (\P) circle (1.3pt) node[above=2pt] {\N};
}
\end{tikzpicture}
\caption{The closed shrinking-diamond chain. Every $a_k$ and $b_k$ has degree three. The dotted segment represents the infinite tail converging to $z$.}
\label{fig:graph}
\end{figure}
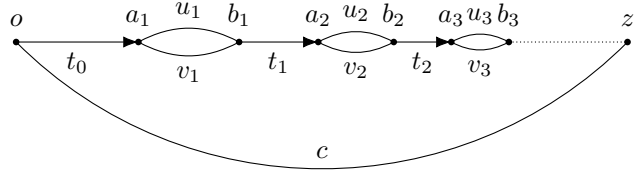

Every $a_k$ and $b_k$ has degree three. The edge lengths form a summable sequence, but the odd-degree vertices are infinite in number.

\begin{lemma}\label{lem:quotient-circle}
There is a $1$-Lipschitz map $\pi:G\to C$ onto a metric circle $C$ of circumference one such that $\pi$ identifies the two parallel edges $u_k$ and $v_k$ isometrically and is injective on every other open edge. In particular, every closed edge of $G$, with the metric induced from $G$, is isometric to an interval of its prescribed length.
\end{lemma}

\begin{proof}
Collapse $u_k$ and $v_k$ onto a single interval of length $\lambda_k$ and leave $c,t_0,t_k$ unchanged. In their natural cyclic order the resulting arcs have total length
\[
  \frac14+\frac14+2\sum_{k=1}^\infty\lambda_k=1,
\]
so their completion is a metric circle $C$ of circumference one. The quotient map $\pi$ is unit speed on each original edge. Hence every rectifiable path in $G$ maps to a path of no greater length in $C$, and therefore $\pi$ is $1$-Lipschitz for the intrinsic path metric on $G$.

Let $x,y$ lie on one closed edge $e$ and let $L_e(x,y)$ denote their arclength separation along $e$. Every edge has length at most $1/4$, so the corresponding arc of $C$ is no longer than a semicircle and
\[
 d_C(\pi x,\pi y)=L_e(x,y).
\]
Since $\pi$ is $1$-Lipschitz,
\[
 L_e(x,y)\le d_G(x,y),
\]
while the subarc of $e$ gives the reverse inequality. Thus equality holds and $e$ is isometrically embedded.
\end{proof}

\begin{remark}\label{rem:no-shortcuts}
The final assertion of Lemma \ref{lem:quotient-circle} is not automatic for a length graph containing cycles: the rest of the graph could shorten the distance between two points of a given edge. The circle quotient rules out precisely these shortcuts. Consequently, the prescribed edge lengths are the intrinsic lengths used both in the computation of \(\Hh^1(G)\) below and in the Euclidean prism charts of \cref{sec:highdim}.
\end{remark}

\begin{proposition}\label{prop:compact-length}
The space $G$ is compact and
\[
      \Hh^1(G)=\frac54.
\]
With its intrinsic path metric, $G$ is a compact length space and hence geodesic.
\end{proposition}

\begin{proof}
By Lemma \ref{lem:quotient-circle}, every edge is isometrically embedded. Its interior is disjoint from the interiors of all other edges, and the vertex/completion set is countable and hence $\Hh^1$-null. Therefore
\[
\begin{aligned}
 \Hh^1(G)
 &=\ell(c)+\ell(t_0)+\sum_{k=1}^\infty
 \bigl(\ell(u_k)+\ell(v_k)+\ell(t_k)\bigr)\\
 &=\frac14+\frac14+3\sum_{k=1}^\infty2^{-k-2}
 =\frac54.
\end{aligned}
\]

Let $G_N$ be the part of the forward chain beginning at $a_N$. Every point of $G_N$ can be joined to $z$ by choosing one of the two diamond edges at each stage and then the connector. Hence
\[
   \sup_{x\in G_N}d_G(x,z)
   \le 2\sum_{k=N}^\infty\lambda_k
   =2^{-N},
\]
and therefore
\[
   \diam(G_N)\le 2^{1-N}\longrightarrow0.
\]
Every sequence either has infinitely many terms in a fixed finite subgraph, where sequential compactness is immediate, or eventually lies in $G_N$ for arbitrarily large $N$ and hence converges to $z$. Thus $G$ is compact. It is a length space by construction; a compact length space is geodesic \cite[Chapter 2]{BBI01}.
\end{proof}

We record a quantitative regularity property.

\begin{proposition}\label{prop:ahlfors}
For every $x\in G$ and every $0<r\le\diam G$,
\[
      r\le \Hh^1(B_G(x,r))\le 4r.
\]
In particular $G$ is $1$-Ahlfors regular and hence doubling.
\end{proposition}

\begin{proof}
Use the quotient map $\pi:G\to C$ from Lemma \ref{lem:quotient-circle}. Outside the countable vertex set, every point of $C$ has at most two preimages. Moreover, the images of distinct open edges are disjoint except that $u_k^\circ$ and $v_k^\circ$ have the same image, and $\pi$ is an isometry on each edge. Decomposing a Borel set $A\subset G$ into its intersections with the open edges therefore gives
\[
      \Hh^1(A)\le 2\Hh^1(\pi(A)).
\]
If $0<r\le1/2$, then by the $1$-Lipschitz property
\[
 \pi(B_G(x,r))\subset B_C(\pi(x),r),
\]
and hence
\[
 \Hh^1(B_G(x,r))
 \le2\Hh^1(B_C(\pi(x),r))
 \le4r.
\]
If $r>1/2$, then
\[
 \Hh^1(B_G(x,r))\le\Hh^1(G)=\frac54<4r.
\]

For the lower bound, if $B_G(x,r)\neq G$, choose $y\notin B_G(x,r)$. Since $G$ is geodesic, a minimizing geodesic from $x$ to $y$ contains an initial subsegment of length $r$ (up to its terminal endpoint) inside the ball, and this subsegment is isometric to an interval. Thus $\Hh^1(B_G(x,r))\ge r$. If $B_G(x,r)=G$, then
\[
  \Hh^1(B_G(x,r))=\Hh^1(G)\ge\diam G\ge r;
\]
the inequality $\Hh^1(G)\ge\diam G$ holds for every connected set of finite $\Hh^1$-measure. This proves the stated Ahlfors bounds. The doubling conclusion is standard; see, e.g., \cite{Heinonen01,AmbrosioTilli04}.
\end{proof}

\section{Boundaryless currents and sharp multiplicity bounds}\label{sec:current}

Orient the edges as follows:
\[
  t_0:o\to a_1,
  \qquad
  u_k,v_k:a_k\to b_k,
  \qquad
  t_k:b_k\to a_{k+1},
\]
and orient the closing edge $c$ from $z$ to $o$. Define
\begin{equation}\label{eq:T}
  T
  =2\llbracket c\rrbracket+2\llbracket t_0\rrbracket
   +\sum_{k=1}^{\infty}
   \left(\llbracket u_k\rrbracket
        +\llbracket v_k\rrbracket
        +2\llbracket t_k\rrbracket\right).
\end{equation}

The coefficient two on each connector is exactly what restores Kirchhoff balance at the trivalent vertices.

\begin{proposition}\label{prop:Tmass}
The current $T$ is integer rectifiable and
\[
      \M(T)=2.
\]
Moreover $\partial T=0$, so $T\in\I_1(G)$.
\end{proposition}

\begin{proof}
By Lemma \ref{lem:quotient-circle}, each oriented edge current is the push-forward of the standard interval current by an isometric parametrization. The open edge images are pairwise disjoint, and the weighted sum of their lengths is finite. Hence the countable representation in \eqref{eq:T} is an integer rectifiable current in the sense of \cite[Section 4]{AK00}, and there is no cancellation in its mass. Thus
\[
\begin{aligned}
 \M(T)
 &=2\ell(c)+2\ell(t_0)
   +\sum_{k=1}^{\infty}
      \bigl(\ell(u_k)+\ell(v_k)+2\ell(t_k)\bigr)\\
 &=\frac12+\frac12+4\sum_{k=1}^{\infty}2^{-k-2}
 =2.
\end{aligned}
\]

Let
\[
 T_N=2\llbracket c\rrbracket+2\llbracket t_0\rrbracket
 +\sum_{k=1}^{N}
 \left(\llbracket u_k\rrbracket+\llbracket v_k\rrbracket
 +2\llbracket t_k\rrbracket\right).
\]
The boundary formula for an oriented interval and telescoping give
\[
      \partial T_N=2\llbracket a_{N+1}\rrbracket-2\llbracket z\rrbracket.
\]
Since $a_{N+1}\to z$, for every Lipschitz function $f$ on $G$,
\[
      \partial T_N(f)=2f(a_{N+1})-2f(z)\longrightarrow0.
\]
Moreover $\M(T-T_N)\to0$, so $T_N(1,f)\to T(1,f)$. Therefore
\[
      \partial T(f)=T(1,f)=\lim_{N\to\infty}\partial T_N(f)=0.
\]
Thus $\partial T=0$. Since $T$ is integer rectifiable with finite mass and its boundary is the zero integer rectifiable current, $T\in\I_1(G)$ by the Ambrosio--Kirchheim definition of integral currents; see \cite[Sections 4 and 8]{AK00}.
\end{proof}

\begin{proposition}\label{prop:setT}
The characteristic set of $T$ is all of $G$:
\[
     \set(T)=G.
\]
\end{proposition}

\begin{proof}
On every edge the absolute multiplicity of $T$ is either one or two. Consequently
\[
      \|T\|\ge \Hh^1\llcorner G.
\]
By Proposition \ref{prop:ahlfors}, for every $x\in G$ and all sufficiently small $r$,
\[
  \frac{\|T\|(B(x,r))}{\omega_1 r}
  \ge \frac{\Hh^1(B(x,r))}{2r}
  \ge\frac12.
\]
Thus every point has positive lower one-dimensional density for $\|T\|$.
\end{proof}

Combining the preceding propositions gives a compact one-dimensional integral current space $(G,d_G,T)$ whose current structure is boundaryless. We next show that both the use of multiplicity two and the geometry of the cycle can be quantified sharply.

Recall that
\[
 \mathscr Z(G)=\{R\in\I_1(G):\partial R=0,\ \set(R)=G\}.
\]

\begin{theorem}[Classification of full-support cycles]\label{thm:cycle-classification}
A current $R$ belongs to $\mathscr Z(G)$ if and only if it has a unique representation
\begin{equation}\label{eq:cycle-classification}
 R=r\left(
       \llbracket c\rrbracket+\llbracket t_0\rrbracket
       +\sum_{k=1}^{\infty}\llbracket t_k\rrbracket
      \right)
   +\sum_{k=1}^{\infty}
      \left(p_k\llbracket u_k\rrbracket
            +q_k\llbracket v_k\rrbracket\right),
\end{equation}
where
\begin{equation}\label{eq:cycle-coefficients}
 r\in\mathbb Z\setminus\{0\},
 \qquad
 p_k,q_k\in\mathbb Z\setminus\{0\},
 \qquad
 p_k+q_k=r
\end{equation}
for every $k$, and
\begin{equation}\label{eq:cycle-summability}
 \sum_{k=1}^{\infty}\lambda_k
       (|p_k|+|q_k|)<\infty.
\end{equation}
For such a current,
\begin{equation}\label{eq:cycle-mass}
       \M(R)=\frac34|r|
       +\sum_{k=1}^{\infty}\lambda_k(|p_k|+|q_k|).
\end{equation}
\end{theorem}

\begin{proof}
Let $R\in\mathscr Z(G)$. On the interior of each edge, the boundary of $R$ vanishes. The interval constancy lemma therefore gives a constant integer coefficient on that edge. No coefficient can vanish, because an interior point of an edge with zero coefficient would not belong to $\set(R)$. Denote the coefficients on $c,t_0,t_k,u_k,v_k$ by $s,r_0,r_k,p_k,q_k$, respectively, using the orientations fixed above.

The vertices $o,a_k,b_k$ are isolated. Evaluating $\partial R=0$ in small star neighborhoods gives the Kirchhoff relations
\[
       s=r_0,
       \qquad
       r_{k-1}=p_k+q_k,
       \qquad
       p_k+q_k=r_k,
\]
where $r_{0}$ denotes the coefficient on $t_0$. Hence all the coefficients $s,r_0,r_1,r_2,\ldots$ have a common nonzero value $r$, and \eqref{eq:cycle-coefficients} follows. Since the open edge interiors are pairwise disjoint, finiteness of $\M(R)$ is equivalent to \eqref{eq:cycle-summability}; it also gives \eqref{eq:cycle-mass}, because
\[
 \ell(c)+\ell(t_0)+\sum_{k=1}^{\infty}\ell(t_k)
 =\frac14+\frac14+\frac14=\frac34.
\]
This proves the necessity and uniqueness of the representation.

Conversely, suppose \eqref{eq:cycle-coefficients}--\eqref{eq:cycle-summability} hold and define $R$ by \eqref{eq:cycle-classification}. It is an integer rectifiable current of finite mass. For the finite truncation
\[
 R_N=r\bigl(\llbracket c\rrbracket+\llbracket t_0\rrbracket\bigr)
 +\sum_{k=1}^{N}
   \left(p_k\llbracket u_k\rrbracket
        +q_k\llbracket v_k\rrbracket
        +r\llbracket t_k\rrbracket\right),
\]
the relations $p_k+q_k=r$ yield
\[
       \partial R_N
       =r\llbracket a_{N+1}\rrbracket-r\llbracket z\rrbracket.
\]
As $a_{N+1}\to z$ and $R_N\to R$ in mass, the same argument as in Proposition \ref{prop:Tmass} gives $\partial R=0$. Finally, every edge coefficient has absolute value at least one, and hence
\[
       \|R\|\ge\Hh^1\llcorner G.
\]
The Ahlfors lower bound in \cref{prop:ahlfors} then gives $\set(R)=G$.
\end{proof}

For $R\in\mathscr Z(G)$, set
\[
       \operatorname{mult}_{\infty}(R)
       :=\operatorname*{ess\,sup}_{x\in G}\theta_R(x).
\]

\begin{corollary}\label{cor:sharp-cycle}
For the graph $G$,
\[
 \min_{R\in\mathscr Z(G)}\operatorname{mult}_{\infty}(R)=2,
 \qquad
 \min_{R\in\mathscr Z(G)}\M(R)=\frac32.
\]
Both minima are attained by
\begin{equation}\label{eq:Rstar}
 R_*=\llbracket c\rrbracket+\llbracket t_0\rrbracket
 +\sum_{k=1}^{\infty}
   \left(2\llbracket u_k\rrbracket
        -\llbracket v_k\rrbracket
        +\llbracket t_k\rrbracket\right).
\end{equation}
\end{corollary}

\begin{proof}
The current $R_*$ corresponds in Theorem \ref{thm:cycle-classification} to
\[
       r=1,
       \qquad p_k=2,
       \qquad q_k=-1.
\]
Thus $R_*\in\mathscr Z(G)$, $\operatorname{mult}_{\infty}(R_*)=2$, and
\[
 \M(R_*)=\frac34+3\sum_{k=1}^{\infty}\lambda_k
 =\frac34+\frac34=\frac32.
\]
On the other hand, Corollary \ref{cor:infinite-odd} rules out a current in $\mathscr Z(G)$ with essential maximal multiplicity one. This proves the first minimum.

For the mass bound, let $R\in\mathscr Z(G)$ have the coefficients from Theorem \ref{thm:cycle-classification}. If $|r|=1$, then two nonzero integers with sum $r$ satisfy
\[
       |p_k|+|q_k|\ge3.
\]
Since $\sum_k\lambda_k=1/4$, \eqref{eq:cycle-mass} gives $\M(R)\ge3/2$. If $|r|\ge2$, the triangle inequality gives $|p_k|+|q_k|\ge|r|$, and therefore
\[
       \M(R)\ge
       \frac34|r|+\frac14|r|=|r|\ge2.
\]
Hence $3/2$ is the global minimum, attained by $R_*$.
\end{proof}

\begin{remark}\label{rem:symmetric-minimizing}
The symmetric cycle $T$ in \eqref{eq:T} corresponds to $r=2$ and $p_k=q_k=1$; it makes the balance law $2=1+1$ transparent. The cycle $R_*$ is asymmetric but mass minimizing. Both show that the sharp essential maximal multiplicity is two.
\end{remark}

\section{Negative answer to Basso's Question 25}\label{sec:negative}

We now apply the parity obstruction.

\begin{theorem}[Negative answer in dimension one]\label{thm:negative1}
For the metric space $G$ constructed in \cref{sec:graph}, there is no current $S\in\I_1(G)$ such that
\[
      \set(S)=G
\]
and the multiplicity of $S$ equals one $\Hh^1$-almost everywhere. Consequently, Basso's Question~25 has a negative answer.
\end{theorem}

\begin{proof}
Suppose such an $S$ exists. Every $a_k$ is an isolated vertex of degree three. Since the multiplicity of $S$ is one almost everywhere on $G$, it is in particular one almost everywhere on all three arms incident to $a_k$. By Theorem \ref{thm:local-parity}, the coefficient of $\partial S$ at $a_k$ is a nonzero odd integer. Hence every $a_k$ belongs to $\spt(\partial S)$.

The points $a_1,a_2,\ldots$ are distinct, so $\partial S$ has infinitely many nonzero integer atoms. This is impossible by Lemma \ref{lem:I0}. Equivalently,
\[
       \M(\partial S)\ge\sum_{k=1}^{\infty}1=\infty,
\]
contradicting $S\in\I_1(G)$.
\end{proof}

\begin{corollary}\label{cor:strong}
Question 25 remains false even after imposing the additional assumptions that the underlying one-dimensional integral current space is compact, geodesic, doubling, $1$-Ahlfors regular, of finite $\Hh^1$-measure, and carries a boundaryless integral current structure.
\end{corollary}

\begin{remark}\label{rem:kirchhoff}
The failure in Theorem \ref{thm:negative1} is caused neither by poor rectifiability nor by boundary in the original current structure: $G$ is geodesic and $1$-Ahlfors regular, and $T$ is a cycle. The obstruction is the integer balance law at the infinitely many odd-valence vertices. At each trivalent vertex $a_k$, the current $T$ satisfies
\[
      2=1+1.
\]
Unit multiplicity cannot satisfy the analogous balance because the sum of three signs $\pm1$ is never zero. Repetition of this local defect at infinitely many vertices forces the boundary mass to diverge.
\end{remark}

\section{Boundaryless counterexamples in every dimension}\label{sec:highdim}

We now show that odd branching survives a closed Euclidean thickening. The use of a torus, rather than a cube, ensures that the current furnishing the integral current space has no boundary.

Fix $n\ge2$, put $m=n-1$, and let
\[
       \mathbb T^m=(\mathbb R/\mathbb Z)^m
\]
be the flat $m$-torus. Set
\[
       X_n=G\times\mathbb T^m
\]
and equip this product with
\[
 d_n\bigl((x,z),(x',z')\bigr)
 =\sqrt{d_G(x,x')^2+d_{\mathbb T^m}(z,z')^2}.
\]

\begin{proposition}\label{prop:product-geometry}
The space $X_n$ is compact, geodesic, doubling, and $n$-Ahlfors regular. More precisely, for
\[
 \mu_n=(\Hh^1\llcorner G)
       \times(\Hh^m\llcorner\mathbb T^m)
\]
there exist constants $0<c_n\le C_n<\infty$ such that
\[
 c_n r^n\le
 \mu_n\bigl(B_{X_n}((x,z),r)\bigr)
 \le C_n r^n
\]
for all $(x,z)\in X_n$ and $0<r\le\diam X_n$. Moreover,
\[
       \mu_n\asymp_n\Hh^n\llcorner X_n.
\]
\end{proposition}

\begin{proof}
Both factors are compact geodesic spaces, and constant-speed minimizing geodesics in the two factors combine to give a minimizing geodesic for the Euclidean product metric. Hence $X_n$ is compact and geodesic.

The flat torus is $m$-Ahlfors regular. Thus there are dimensional constants $a_m,b_m>0$ such that
\[
 a_m s^m\le
 \Hh^m(B_{\mathbb T^m}(z,s))
 \le b_m s^m
\]
for $0<s\le\diam\mathbb T^m$. Together with \cref{prop:ahlfors} and the inclusions
\[
\begin{aligned}
 B_G\!\left(x,\frac r{\sqrt2}\right)
 \times
 B_{\mathbb T^m}\!\left(z,\frac r{\sqrt2}\right)
 &\subset B_{X_n}((x,z),r),\\
 B_{X_n}((x,z),r)
 &\subset B_G(x,r)\times B_{\mathbb T^m}(z,r),
\end{aligned}
\]
this gives the asserted estimates for all sufficiently small $r$. Compactness and monotonicity of balls extend them, after changing the constants, to $0<r\le\diam X_n$.

The open sets $e^\circ\times U$, where $e$ is an edge and $U$ is a Euclidean coordinate cube in $\mathbb T^m$, give countably many bi-Lipschitz $n$-rectifiable charts covering $\mu_n$-almost all of $X_n$. The standard comparison of an Ahlfors regular measure with Hausdorff measure yields $\mu_n\asymp_n\Hh^n\llcorner X_n$; see \cite[Chapter~8]{Heinonen01} or \cite{AmbrosioTilli04}. Doubling follows from Ahlfors regularity.
\end{proof}

Let $Q=[0,1]^m$, let $q:Q\to\mathbb T^m$ be the quotient identifying opposite faces, and set
\[
       F=\operatorname{id}_G\times q:G\times Q\longrightarrow X_n.
\]
Iterating Wenger's interval-product construction and applying the canonical factor-flip isometries \cite[Section 3.1, Theorem 3.2]{Wenger07}, define
\begin{equation}\label{eq:torus-product-current}
       T_n=F_\#\bigl(T\times\llbracket Q\rrbracket\bigr).
\end{equation}
The quotient formulation is useful because it gives the boundary cancellation directly from the interval-product formula.

\begin{proposition}\label{prop:product-current}
The current $T_n$ defined in \eqref{eq:torus-product-current} belongs to $\I_n(X_n)$ and satisfies
\[
       \partial T_n=0,
       \qquad
       \set(T_n)=X_n.
\]
\end{proposition}

\begin{proof}
The iterated interval product $T\times\llbracket Q\rrbracket$ is integral and has finite mass and finite boundary mass. Since $\partial T=0$, Wenger's boundary formula
\[
 \partial([0,1]\times R)
 =[1]\times R-[0]\times R-[0,1]\times\partial R
\]
shows that its boundary is the alternating sum of the $2m$ currents carried by the faces of $G\times Q$. Under $F$, each pair of opposite faces has the same image in $G\times\mathbb T^m$ and opposite induced orientations. The pairs therefore cancel, and naturality of the boundary under push-forward gives
\[
       \partial T_n
       =F_\#\partial\bigl(T\times\llbracket Q\rrbracket\bigr)=0.
\]
Thus $T_n\in\I_n(X_n)$.

Let $\Sigma\subset\mathbb T^m$ be the image under $q$ of the boundary faces of $Q$. It is $\Hh^m$-null. On every product chart
\[
       e^\circ\times U,
       \qquad U\Subset\mathbb T^m\setminus\Sigma,
\]
the map $F$ is locally an isometry and $T_n$ is the ordinary integration current with coefficient equal to that of $T$ on $e$, namely one or two. These charts cover $\mu_n$-almost all of $X_n$. Hence, for a dimensional constant $a_n>0$,
\[
       \|T_n\|\ge a_n\mu_n.
\]
By Proposition \ref{prop:product-geometry}, every sufficiently small ball centered at any point of $X_n$ has $\mu_n$-measure at least $c_nr^n$. It follows that
\[
 \|T_n\|(B_{X_n}((x,z),r))\ge a_nc_nr^n,
\]
so every point has positive lower $n$-density for $\|T_n\|$. Therefore $\set(T_n)=X_n$.
\end{proof}

We next choose a global slicing map with regular finite fibers. Define
\[
 h:\mathbb R/\mathbb Z\longrightarrow[-1,1],
 \qquad h([t])=\cos(2\pi t),
\]
and set
\[
 \psi:\mathbb T^m\longrightarrow[-1,1]^m,
 \qquad
 \psi(z_1,\ldots,z_m)=(h(z_1),\ldots,h(z_m)).
\]
Finally, let
\begin{equation}\label{eq:slicing-map}
 \Pi:X_n\longrightarrow\mathbb R^m,
 \qquad
 \Pi(x,z)=\psi(z).
\end{equation}
For every $y\in(-1,1)^m$, the set
\[
       Z_y:=\psi^{-1}(y)
\]
consists of exactly $2^m$ points, and $\psi$ has full rank at each point of $Z_y$.

\begin{lemma}[Unit multiplicity on the fiber components]\label{lem:slice-unit}
Suppose $S\in\I_n(X_n)$ satisfies
\[
       \set(S)=X_n,
       \qquad
       \theta_S=1\quad\Hh^n\text{-a.e. on }X_n.
\]
Then there is a Borel set $Y_S\subset(-1,1)^m$ of full
$\mathcal L^m$-measure such that, for every $y\in Y_S$, the slice admits a
finite decomposition
\[
       \langle S,\Pi,y\rangle
       =\sum_{z\in Z_y}R_{y,z},
\]
where each $R_{y,z}\in\I_1(X_n)$ is concentrated on
$G\times\{z\}$. After identifying this closed isometric fiber with $G$,
the current $R_{y,z}$ has absolute multiplicity one $\Hh^1$-almost
everywhere on the interior of every edge.
\end{lemma}

\begin{proof}
By Ambrosio--Kirchheim's slicing and boundary theorems, there is a conull
set $Y_0\subset\mathbb R^m$ such that
\[
       R_y:=\langle S,\Pi,y\rangle\in\I_1(X_n)
       \qquad (y\in Y_0),
\]
and
\begin{equation}\label{eq:slice-carrier}
       \spt R_y\subset \Pi^{-1}(y).
\end{equation}
See \cite[Theorems~5.6--5.7]{AK00}. We next make the multiplicity statement
precise.

Since $X_n$ is compact and separable, the countable
Fr\'echet--Kuratowski map gives an isometric embedding
\[
       \iota:X_n\longrightarrow \ell^\infty=(\ell^1)^*;
\]
see \cite{BBI01}. The image $\iota(X_n)$ is compact and hence closed.
Extend each coordinate of $\Pi\circ\iota^{-1}$ from $\iota(X_n)$ to
$\ell^\infty$ by McShane's theorem \cite{McShane34}. Applying
Ambrosio--Kirchheim's precise slice representation in a dual space with
separable predual to $\iota_\#S$, and then identifying $\iota(X_n)$ with
$X_n$, gives, for almost every $y\in\mathbb R^m$, an orientation $\tau_y$
such that
\begin{equation}\label{eq:precise-slice}
 R_y
 =\llbracket
      \set(S)\cap\Pi^{-1}(y),\theta_S,\tau_y
   \rrbracket.
\end{equation}
Here the right-hand side is the standard bracket notation for the current
carried by the indicated rectifiable set, with integer multiplicity
$\theta_S$ and orientation $\tau_y$.
This is exactly \cite[Theorem~9.7]{AK00}; in particular, the integer
multiplicity appearing in a slice is the restriction of the multiplicity
of $S$, not a new weight involving the coarea Jacobian.

It remains to check carefully that the exceptional $\Hh^n$-null set on
which $\theta_S$ may differ from one does not leave positive-length pieces
in the fibers under consideration. Put
\[
       N_S=\{w\in X_n:\theta_S(w)\ne1\};
       \qquad \Hh^n(N_S)=0.
\]
For every open edge $e^\circ$ and every one of the $2^m$ inverse branches
$z_\sigma:(-1,1)^m\to\mathbb T^m$ of $\psi$, apply the ordinary coarea
formula \cite[3.2.22]{Federer69} in the Euclidean product charts
\[
       e^\circ\times z_\sigma(K),
       \qquad K\Subset(-1,1)^m.
\]
These are legitimate rectifiable charts because every edge is intrinsically
isometric by \cref{lem:quotient-circle}, and the differential of $\psi$ has
rank $m$ on each branch. Since $N_S$ is $\Hh^n$-null, coarea gives
\begin{equation}\label{eq:null-fiber-sections}
 \Hh^1\bigl(
    N_S\cap(e^\circ\times\{z_\sigma(y)\})
 \bigr)=0
 \quad\text{for $\mathcal L^m$-a.e. }y\in K.
\end{equation}
Exhaust $(-1,1)^m$ by countably many compact boxes $K$, and take the union
of the exceptional sets in \eqref{eq:null-fiber-sections} over the
countably many edges and the finitely many branches. We obtain a conull set
$Y_1\subset(-1,1)^m$ on which \eqref{eq:null-fiber-sections} holds
simultaneously for every edge and every branch.

Let $Y_S$ be the intersection of $Y_0\cap Y_1$ with the conull set on which
\eqref{eq:precise-slice} holds, and fix $y\in Y_S$. Since
$\set(S)=X_n$, equations \eqref{eq:precise-slice} and
\eqref{eq:null-fiber-sections} imply that, on every
$e^\circ\times\{z\}$ with $z\in Z_y$, the slice $R_y$ has absolute
multiplicity one $\Hh^1$-almost everywhere.

We finally separate the finitely many fiber components without losing
integrality. Put $D_z=G\times\{z\}$ and
\[
       \delta_y=\frac15
       \min_{\substack{z,z'\in Z_y\\ z\ne z'}}
       d_{\mathbb T^m}(z,z')>0.
\]
Choose a Lipschitz function $\chi:[0,\infty)\to[0,1]$ which is one on
$[0,1]$ and zero on $[2,\infty)$, and define
\[
       \eta_z(w)=\chi\!\left(
          \frac{d_{X_n}(w,D_z)}{\delta_y}
       \right).
\]
Then $\eta_z$ is identically one in a neighborhood of $D_z$ and identically
zero in neighborhoods of all $D_{z'}$, $z'\ne z$. By
\eqref{eq:slice-carrier}, $\eta_z$ is therefore locally constant near
$\spt R_y$. Set
\[
       R_{y,z}=R_y\llcorner\eta_z.
\]
On the rectifiable carrier of $R_y$ the multiplier $\eta_z$ takes only the
integer values zero and one, so $R_{y,z}$ is integer rectifiable and is
concentrated on $D_z$. Moreover,
$\spt(\partial R_y)\subset\spt R_y$ and locality gives
$R_y\llcorner d\eta_z=0$. The restriction boundary formula therefore yields
\begin{equation}\label{eq:component-boundary}
 \partial R_{y,z}
   =(\partial R_y)\llcorner\eta_z
     -R_y\llcorner d\eta_z
   =(\partial R_y)\llcorner\eta_z.
\end{equation}
The last current is an integer rectifiable $0$-current because $\eta_z$ is
again zero or one on $\spt(\partial R_y)$. Hence
$R_{y,z}\in\I_1(X_n)$. On $\spt R_y$ the finitely many functions
$\eta_z$, $z\in Z_y$, form a partition into the fiber components, and thus
\[
       R_y=\sum_{z\in Z_y}R_{y,z}.
\]
The edgewise unit-multiplicity assertion follows from
\eqref{eq:precise-slice}--\eqref{eq:null-fiber-sections}. Finally, the
closed-subspace identification recalled in \cref{sec:prelim} identifies
each component with an integral current on $G$.
\end{proof}

\begin{remark}[The precise role of slicing]\label{rem:slicing}
Theorem 5.7 of \cite{AK00} supplies integrality of the slice, whereas
Theorem 9.7 supplies the exact carrier--multiplicity representation
\eqref{eq:precise-slice}. Coarea is used only to pass the
$\Hh^n$-exceptional set for $\theta_S$ to an $\Hh^1$-exceptional set on
almost every fiber. The positive separation of the finitely many level-set
components and \eqref{eq:component-boundary} then isolate one copy of $G$
as an integral current. No assertion that
$\set(R_{y,z})=G\times\{z\}$ is needed for the parity argument.
\end{remark}

\begin{theorem}[Boundaryless negative examples in every dimension]\label{thm:negative-n}
For every $n\ge1$ there exist a compact, geodesic, doubling, $n$-Ahlfors regular metric space $X_n$ and an integral current $T_n\in\I_n(X_n)$ such that
\[
       \partial T_n=0,
       \qquad
       \set(T_n)=X_n,
\]
but there is no $S\in\I_n(X_n)$ satisfying
\[
       \set(S)=X_n,
       \qquad
       \theta_S=1\quad\Hh^n\text{-a.e. on }X_n.
\]
\end{theorem}

\begin{proof}
For $n=1$, take $(G,T)$ and apply Theorem \ref{thm:negative1}. Let $n\ge2$. The geometric conclusions follow from Proposition \ref{prop:product-geometry}, while Proposition \ref{prop:product-current} gives a boundaryless integral current structure on $X_n$.

Suppose that a current $S$ with the asserted unit multiplicity exists. Choose $y\in(-1,1)^m$ for which Lemma \ref{lem:slice-unit} holds and choose any $z\in Z_y$. After identifying $G\times\{z\}$ with $G$, the component $R_{y,z}$ is an integral $1$-current having unit multiplicity on the interior of every edge. This contradicts Corollary \ref{cor:infinite-odd}, since the vertices $a_k\times\{z\}$ are all trivalent in the fiber graph.
\end{proof}

\section*{Acknowledgements}
This work was supported by Guangdong Basic and Applied Basic Research Foundation
(No.~2022A1515110967 and No.~2023A1515011809). The author acknowledges the use of AI-assisted tools in manuscript preparation and assumes full responsibility for all mathematical statements and proofs.

\end{document}